\documentclass[article]{elsarticle}

\usepackage{hyperref}
\usepackage{geometry}
\usepackage{amsmath,amssymb,enumerate,amsfonts}% »¨Ìå¹«Ê½

\usepackage[normalem]{ulem} % use normalem to protect \emph
\usepackage{longtable,threeparttable,float}
\usepackage[table]{xcolor}
\usepackage{graphicx, subfigure}
\usepackage{caption}
\usepackage{algorithm}%
\usepackage{algorithmicx}%
\usepackage{algpseudocode}%
\biboptions{sort&compress}
\usepackage{epstopdf}
\usepackage{tabularx}
\usepackage[justification=centering]{caption}

\usepackage{setspace}

\usepackage[figuresright]{rotating}

\journal{Journal of \LaTeX\ Templates}

\newtheorem{definition}{Definition}[section]
\newtheorem{remark}{Remark}[section]

\newtheorem{theorem}{Theorem}[section]
\newtheorem{proposition}{Proposition}[section]
\newtheorem{lemma}{Lemma}[section]

\newenvironment{proof}{{\noindent \it Proof.}}{\hfill $\square$\par}

\begin{document}

\begin{frontmatter}

\title{Improvements to steepest descent method for multi-objective optimization}

\author[mymainaddress,mysecondaryaddress]{Wang Chen}
\ead{chenwangff@163.com}

\author[mymainaddress]{Liping Tang}
\ead{tanglps@163.com}

\author[mymainaddress]{Xinmin Yang}
\ead{xmyang@cqnu.edu.cn}

\address[mymainaddress]{National Center for Applied Mathematics in Chongqing, Chongqing Normal University, Chongqing, 401331, China}
\address[mysecondaryaddress]{School of Mathematical Sciences, University of Electronic Science and Technology of China, Chengdu, 611731, China}

\begin{abstract}
In this paper, we propose a simple yet efficient strategy for improving the multi-objective steepest descent method proposed by Fliege and Svaiter (Math Methods Oper Res, 2000, 3: 479--494). The core idea behind this strategy involves incorporating a positive modification parameter into the iterative formulation of the multi-objective steepest descent algorithm in a multiplicative manner. 
This modification parameter captures certain second-order information associated with the objective functions. We provide two distinct methods for calculating this modification parameter, leading to the development of two improved multi-objective steepest descent algorithms tailored for solving multi-objective optimization problems. Under reasonable assumptions, we demonstrate the convergence of sequences generated by the first algorithm toward a critical point. Moreover, for strongly convex multi-objective optimization problems, we establish the linear convergence to Pareto optimality of the sequence of generated points. The performance of the new algorithms is empirically evaluated through a computational comparison on a set of multi-objective test instances. The numerical results underscore that the proposed algorithms consistently outperform the original multi-objective steepest descent algorithm.
\end{abstract}


\end{frontmatter}

%\linenumbers
\section{Introduction}

Let us first consider the following single-objective optimization problem:
\begin{equation}\label{sop}
	\min_{x\in\mathbb{R}^{n}}\quad f(x),\\
\end{equation}
where $f:\mathbb{R}^{n}\rightarrow\mathbb{R}$ is continuously differentiable. One of the earliest and most well-known methods for solving problem \eqref{sop} is the gradient descent method, also known as the steepest descent (SD) method, designed by Cauchy \cite{cauchy1848method} in 1847. The algorithm begins with an arbitrarily chosen initial point $x^{0}\in{\rm dom} f$ and generates a sequence of iterates using the following rule:
\begin{equation}\label{sd_iterative}
	x^{k+1}=x^{k}+t_{k}d^{k},~k=0,1,2,\ldots,
\end{equation}
where $t_{k}>0$ is the stepsize and $d^{k}=-\nabla f(x^{k})$ is the negative gradient of $f$.  The SD method behaves poorly except for very well-conditioned problems. To improve the behavior of the SD algorithm, Raydan and Svaiter \cite{raydan2002relaxed} introduced a relaxation parameter denoted as $\theta_{k}$, with values ranging between 0 and 2, into the iterative formulation represented by \eqref{sd_iterative}. Notably, the selection of $\theta_{k}$ is performed randomly during the iterative process. This innovative modification gives rise to a randomly relaxed SD method tailored for the resolution of quadratic optimization problems. In order to alleviate the constraints of the objective function, Andrei \cite{andrei2006acceleration} devised a novel technique to  determine the parameter $\theta_{k}$. The resulting algorithm in \cite{andrei2006acceleration} was proven to be linearly convergent, exhibiting a significant improvement in the reduction of the function value. Subsequently, the author in \cite{stanimirovic2010accelerated} presented a new method for choosing the parameter $\theta_{k}$, wherein an appropriate diagonal matrix was employed to approximate the Hessian of the objective function. Numerical justifications demonstrating the superiority of the improved SD algorithms over the standard SD algorithm were provided in \cite{raydan2002relaxed, andrei2006acceleration, stanimirovic2010accelerated}. The similar idea was also applied in the research works of  \cite{petrovic2018hybridization,ivanov2021accelerated,ivanov2023accelerated}.

Consider now the following multi-objective optimization problem:
\begin{equation}\label{mop}
	\min_{x\in\mathbb{R}^{n}}\quad F(x)=(F_{1}(x),F_{2}(x),...,F_{m}(x))^{\top},\\
\end{equation}
where $F:\mathbb{R}^{n}\rightarrow\mathbb{R}^{m}$ is a continuously differentiable vector-valued function. The objective functions in \eqref{mop} are usually conflicting, so a unique solution optimizing all the objective functions simultaneously does not exist. Instead, there exists a set of points known as the Pareto optimal solutions, which
are characterized by the fact that an improvement in one objective will result in a deterioration in at least one of the other objectives. Applications of this type of problem can be found in various areas, such as engineering \cite{R_m2013}, finance \cite{Z_m2015}, environment analysis \cite{F_o2001}, management science \cite{T_a2010}, machine learning \cite{J_m2006}, and so on. 

Over the past two decades, a class of methods extensively studied for solving problem \eqref{mop} is the {descent algorithms}, which still follow the iterative scheme \eqref{sd_iterative}. The stepsize $t_{k}>0$ is obtained by a strategy in the multi-objective sense, and $d^{k}$ is a common descent direction for all objective functions $F_{1},F_{2},\ldots,F_{m}$. The suitable settings of $t_{k}$ and $d^{k}$ can ensure that the value of the objective functions decreases at each iteration in the partial order induced by the natural cone. 
A variety of multi-objective stepsize strategies for determining $t_{k}$ have been  proposed:
Armijo \cite{fliege2000steepest,fliege2009newton}, (strong) Wolfe \cite{lucambio2018nonlinear,goncalves2022study}, Goldstein \cite{wang2019extended}, nonmonotone \cite{mita2019nonmonotone,chen2023conditional}, etc. The descent direction $d^{k}$ in multi-objective optimization needs to be obtained by solving an auxiliary subproblem at $x^{k}$. To our knowledge, there are several strategies for choosing the search direction $d^{k}$ 
(see \cite{fliege2000steepest,fliege2009newton,wang2019extended,qu2011quasi,lucambio2018nonlinear,cocchi2020convergence,chen2023memory,ansary2015modified,goncalves2022study} and references therein). A representative descent algorithm is Fliege and Svaiter's \cite{fliege2000steepest} research work in 2000, where $t_{k}$ satisfies the multi-objective Armijo rule and $d^{k}$ is generated by solving the strongly convex quadratic problem (see \eqref{sub_pro}) or its dual problem (see \eqref{dual} and \eqref{dual_sol}). This yields the multi-objective steepest descent (MSD) method. Under mild conditions, it was shown that the MSD algorithm converges to a point satisfying certain first-order necessary condition for Pareto optimality. As reported in \cite{mita2019nonmonotone,morovati2018barzilai}, it does not, however, perform satisfactorily in practical computations, possibly due to the small stepsize.

The purpose of this paper is to take a step further on the direction of Fliege and Svaiter's \cite{fliege2000steepest} work. Drawing inspiration from the works in \cite{raydan2002relaxed,andrei2006acceleration}, we propose an improved scheme for the MSD algorithm to solve problem \eqref{mop}, utilizing the following iterative form:
\begin{equation}\label{new_iter}
	x^{k+1}=x^{k}+\theta_{k} t_{k}v(x^{k}),
\end{equation}
where $v(x^{k})$ represents the multi-objective descent direction (see \eqref{opt_sol} or \eqref{dual_sol}), and $\theta_{k}>0$ is a modification parameter that accumulates second-order information about the objective functions around $x^{k}$. We present two different approaches for calculating the parameter $\theta_{k}$. The basic idea behind the first method
is to approximate the Hessian of each objective function using a common diagonal matrix. It is noteworthy that the authors in \cite{el2021accelerated} also employed a similar idea to introduce a multi-objective diagonal steepest descent method for solving problem \eqref{mop}. In the second scheme, we initially apply the MSD algorithm to obtain an iterative point $z^{k}$ and then use the second-order information of the objective function at the points $x^{k}$ and $z^{k}$ to update the parameter $\theta_{k}$. Consequently, we introduce two improved MSD algorithms for solving problem \eqref{mop}. Under suitable assumptions, we prove that the first algorithm converges to points that satisfy the first-order necessary conditions for Pareto optimality. Finally, numerical experiments on several multi-objective optimization problems illustrating the practical behavior of the proposed approaches are presented, and comparisons with the MSD algorithm \cite{fliege2000steepest} and the multi-objective diagonal steepest descent algorithm \cite{el2021accelerated} are discussed.

The outline of this paper is as follows. Section
\ref{sec:2} provides some basic definitions, notations and auxiliary results. Section \ref{sec:3} gives the first improvement scheme and analyzes its convergence. Section \ref{sec:4} presents the second improvement scheme. Section \ref{sec:5} includes numerical experiments to demonstrate the performance of the proposed algorithms. Finally, in Section \ref{sec:6}, some conclusions and remarks are given.

\section{Preliminaries}\label{sec:2}

Throughout this paper, we use $\langle \cdot, \cdot \rangle$ and $\|\cdot\|$ to represent, respectively, the standard inner product and norm in $\mathbb{R}^{n}$. The notations $\mathbb{R}_{+}^{n}$ and $\mathbb{R}_{++}^{n}$ are employed to denote the non-negative orthant and positive orthant of $\mathbb{R}^{n}$, respectively. For any positive integer $m$, we define $\langle m \rangle = \{1, 2, \ldots, m\}$. As usual, for $x,y\in\mathbb{R}^{n}$, we use the symbol $\preceq$ to denote the classical partial order defined by
\begin{equation*}
	x\preceq y ~\Leftrightarrow ~y-x\in\mathbb{R}_{+}^{n}.
\end{equation*}

\begin{definition}\setstretch{1.25}\normalfont\cite{miettinen1999nonlinear}
	A point $\bar{x}\in\mathbb{R}^{n}$ is said to be \emph{Pareto optimal} of problem \eqref{mop} if there exists no $x\in\mathbb{R}^{n}$ such that $F(x)\preceq F(\bar{x})$ and $F(x)\neq F(\bar{x})$.
\end{definition}

A first-order necessary condition introduced in \cite{fliege2000steepest} for Pareto optimality of a point $\bar{x}\in\mathbb{R}^{n}$ is
\begin{equation}\label{fir_ord_opt}
	JF(\bar{x})(\mathbb{R}^{n})\cap(-\mathbb{R}^{m}_{++})=\emptyset,
\end{equation}
where $ JF(\bar{x})(\mathbb{R}^{n})=\{JF(\bar{x})d:d\in\mathbb{R}^{n}\}$ and 
$$JF(\bar{x})d=(\langle\nabla F_{1}(\bar{x}),d\rangle,\langle\nabla F_{2}(\bar{x}),d\rangle,\ldots,\langle\nabla F_{m}(\bar{x}),d\rangle)^{\top}.$$

A point $\bar{x}\in\mathbb{R}^{n}$ that satisfies the relation \eqref{fir_ord_opt} is referred to as a \emph{Pareto critical} point (see \cite{fliege2000steepest}). Alternatively, for any $d\in\mathbb{R}^{n}$, there exists $i^{*}\in\langle m\rangle$ such that $$(JF(\bar{x})d)_{i^{*}}=\langle\nabla F_{i^{*}}(\bar{x}),d\rangle\geq0,$$
which implies $\max_{i\in\langle m\rangle}\langle\nabla F_{i}(\bar{x}),d\rangle\geq0$ for any $d\in\mathbb{R}^{n}$. 
If $x\in\mathbb{R}^{n}$ is not a Pareto critical point, then there exists a vector $d\in\mathbb{R}^{n}$ satisfying $JF(x)d\in-\mathbb{R}^{m}_{++}$. We refer to the vector $d$ as a \emph{descent direction} for $F$ at $x$.

Define the function $\psi:\mathbb{R}^{n}\times\mathbb{R}^{n}\rightarrow \mathbb{R}$ by
\begin{equation}\label{h}
	\psi(x,d)=\max_{i\in\langle m\rangle}\langle\nabla F_{i}(x),d\rangle.
\end{equation}
Clearly, at the point $x$, $\psi$ is convex and positively homogeneous (i.e., $\psi(x,\varrho z)=\varrho \psi(x,z)$ for all $\varrho>0$ and $z\in\mathbb{R}^{n}$). Let us now consider the following scalar optimization problem:
\begin{equation}\label{sub_pro}
	\min_{d\in\mathbb{R}^{n}}\:\psi(x,d)+\frac{1}{2}\|d\|^{2}.
\end{equation}
The objective function in \eqref{sub_pro} is proper, closed and strongly convex. Therefore, problem \eqref{sub_pro} admits a unique optimal solution, referred to as the \emph{steepest descent direction} (see \cite{fliege2000steepest}). Denote the optimal solution of (\ref{sub_pro}) by $v(x)$, i.e.,
\begin{equation}\label{opt_sol}
	v(x)=\mathop{\rm argmin}_{d\in\mathbb{R}^{n}}\: \psi(x,d)+\frac{1}{2}\|d\|^{2},
\end{equation}
and let the optimal value of (\ref{sub_pro}) be defined as $\theta(x)$, i.e., 
\begin{equation}\label{opt_val}
	\gamma(x)=\psi(x,v(x))+\frac{1}{2}\|v(x)\|^{2}.
\end{equation}

\begin{remark}\setstretch{1.25}\normalfont
	Observe that in scalar optimization (i.e., $m=1$), we obtain $\psi(x,d)= \langle\nabla F_{1}(x),d\rangle$, $v(x)=-\nabla F_{1}(x)$ and $\gamma(x)=-\|\nabla F_{1}(x)\|^{2}/2$.
\end{remark}

To determine  $v(x)$, one can consider the corresponding dual problem of \eqref{sub_pro} (see \cite{fliege2000steepest}):

\begin{equation}\label{dual}
	\begin{aligned}
		\min&\quad \dfrac{1}{2}\left\|\sum_{i=1}^{m}\lambda_{i}\nabla F_{i}(x)\right\|^{2}\\
		{\rm s.t.}&\quad \lambda\in\varLambda^{m},
	\end{aligned}
\end{equation}
where $\varLambda^{m}=\{\lambda\in\mathbb{R}^{m}:\sum_{i=1}^{m}\lambda_{i}=1,\lambda_{i}\geq0,\forall i\in\langle m\rangle\}$ stands for the simplex set. Then, $v(x)$ can also be represented as
\begin{equation}\label{dual_sol}
	v(x)=-\sum_{i=1}^{m}\lambda_{i}(x)\nabla F_{i}(x),
\end{equation}
where $\lambda(x)=(\lambda_{1}(x),\lambda_{2}(x),\ldots,\lambda_{m}(x))$ is optimal solution of $\eqref{dual}$.

We finish this section with the following results, which will be used in our subsequent analysis.

\begin{proposition}\setstretch{1.25}{\rm\cite{fliege2000steepest}}\label{pa_sta_equ}
	Let $v(\cdot)$ and $\gamma(\cdot)$ be as in \eqref{opt_sol} and \eqref{opt_val}, respectively. The following statements hold:
	\begin{enumerate}[{\rm(i)}]\setlength{\itemsep}{-0.0in}
		\item if $x$ is a Pareto critical point of problem \eqref{mop}, then $v(x)= 0$ and $\gamma(x)=0$;
		\item if $x$ is not a Pareto critical point of problem \eqref{mop}, then $v(x)\neq 0$, $\gamma(x)<0$ and $\psi(x,v(x))<-\|v(x)\|^{2}/2<0$;
		\item $v(\cdot)$ is continuous.
	\end{enumerate}
\end{proposition}

\section{The first improvement scheme}\label{sec:3}

This section provides the first method for determining  the modification parameter $\theta_{k}$, which uses the approximation of the Hessian of the objective function by means of the diagonal matrix whose entries are appropriately computed.

 For the sake of simplicity, we abbreviate $\lambda_{i}(x^{k})$ as $\lambda_{i}^{k}$, $i\in\langle m\rangle$. In \eqref{new_iter}, we let 
\begin{equation}\label{theta1}
	\theta_{k}=\tau^{-1}_{k},
\end{equation}
then
$
x^{k+1}=x^{k}+t_{k}\tau^{-1}_{k}v(x^{k}).
$ If $F$ is twice continuously differentiable, then using the second-order Taylor expansion of $F_{i}$ at the point $x^{k+1}$, we get 
\begin{equation}\label{dsd_second_order}
	\begin{aligned}
		F_{i}(x^{k+1})= F_{i}(x^{k})+t_{k}\tau^{-1}_{k}\nabla F_{i}(x^{k})^{\top}v(x^{k})
		+\dfrac{1}{2}t_{k}^{2}\tau_{k}^{-2}v(x^{k})^{\top}\nabla^{2}F_{i}(\zeta)v(x^{k}),
	\end{aligned}
\end{equation}
for all $i\in\langle m\rangle$, where $\zeta\in\mathbb{R}^{n}$ lies in the line segment connecting $x^{k}$ and $x^{k+1}$, and it is defined by
\begin{equation*}
	\zeta=x^{k}+\xi(x^{k+1}-x^{k}),\quad\xi\in[0,1].
\end{equation*}
Considering the distance between $x^{k}$ and $x^{k+1}$ is small enough (using
the local character of searching), we can take $\xi=1$ and this results in $\zeta=x^{k+1}$. Multiplying by $\lambda_{i}^{k}$ and summing over  $i\in\langle m\rangle$ in \eqref{dsd_second_order}, one has
\begin{equation}\label{dsd_second_order1}
	\begin{aligned}
		\sum_{i=1}^{m}\lambda_{i}^{k}F_{i}(x^{k+1})=& \sum_{i=1}^{m}\lambda_{i}^{k}F_{i}(x^{k})+t_{k}\tau_{k}^{-1}\left(\sum_{i=1}^{m}\lambda_{i}^{k}\nabla F_{i}(x^{k})\right)^{\top}v(x^{k})\\
		&+\dfrac{1}{2}t_{k}^{2}\tau_{k}^{-2}v(x^{k})^{\top}\left(\sum_{i=1}^{m}\lambda_{i}^{k}\nabla^{2}F_{i}(x^{k+1})\right)v(x^{k})\\
		=& \sum_{i=1}^{m}\lambda_{i}^{k}F_{i}(x^{k})-t_{k}\tau_{k}^{-1}\|v(x^{k})\|^{2}\\
		&+\dfrac{1}{2}t_{k}^{2}\tau_{k}^{-2}v(x^{k})^{\top}\left(\sum_{i=1}^{m}\lambda_{i}^{k}\nabla^{2}F_{i}(x^{k+1})\right)v(x^{k}),
	\end{aligned}
\end{equation}
where the second equality follows from \eqref{dual_sol} with $x=x^{k}$. If we consider replacing the Hessian $\sum_{i=1}^{m}\lambda_{i}^{k}\nabla^{2}F_{i}(x^{k+1})$ in \eqref{dsd_second_order1} with $\tau_{k+1}I$, then
\begin{equation*}
	\begin{aligned}
		\sum_{i=1}^{m}\lambda_{i}^{k}F_{i}(x^{k+1})=
		\sum_{i=1}^{m}\lambda_{i}^{k}F_{i}(x^{k})-t_{k}\tau_{k}^{-1}\|v(x^{k})\|^{2}
		+\dfrac{1}{2}t_{k}^{2}\tau_{k+1}\tau_{k}^{-2}\|v(x^{k})\|^{2}.
	\end{aligned}
\end{equation*}
Directly following the aforementioned relation, we obtain
\begin{equation}\label{tau_k}			
	\tau_{k+1}=\dfrac{2\tau_{k}\left(\tau_{k}\displaystyle\sum_{i=1}^{m}\lambda^{k}_{i}(F_{i}(x^{k+1})-F_{i}(x^{k}))+t_{k}\|v(x^{k})\|^{2}\right)}{t_{k}^{2}\|v(x^{k})\|^{2}}.
\end{equation}

\begin{remark}\setstretch{1.25}\normalfont
	For $k=0$, we let $\tau_{k}=1$.	In general, it is essential to note that ensuring the non-negativity of the parameter $\tau_{k+1}$ is not always guaranteed. In cases where $\tau_{k+1}<0$, we take the corrective action of setting $\tau_{k+1}=1$. The rationale behind this choice is as follows: if $\sum_{i=1}^{m}\lambda_{i}^{k}\nabla^{2}F_{i}(x^{k+1})$ is a non-positive definite matrix, selecting $\tau_{k+1}=1$ leads to the next iterative point being effectively computed using the MSD algorithm, i.e., $x^{k+2}=x^{k+1}+t_{k}v(x^{k})$. Therefore, for all $k\geq0$, we have $\tau_{k}>0$.
\end{remark}

\begin{remark}\setstretch{1.25}\normalfont
	Notice that the authors in \cite{el2021accelerated} also utilize $\tau_{k+1}I$ to approximate $\sum_{i=1}^{m}\lambda_{i}^{k}\nabla^{2}F_{i}(x^{k+1})$ and subsequently determine the parameter $\tau_{k+1}$ by considering the modified secant equation. The specific value for the parameter involves the difference between adjacent iterative points and the gradient difference between these points, expressed as:
	\begin{equation*}
		\tau_{k+1}=\max\left\{\tau_{0},\dfrac{(\sum_{i=1}^{m}\lambda_{i}^{k}(\nabla F_{i}(x^{k+1})-\nabla F_{i}(x^{k})))^{\top}(x^{k+1}-x^{k})}{\|x^{k+1}-x^{k}\|_{2}^{2}}\right\},
	\end{equation*}
	where $\tau_{0}>0$ is a given initial parameter.
\end{remark}

In this paper, we consider the iterative stepsize $t_{k}$ using the Armijio line search technique, which is presented as follows: 

\noindent\textbf{Armijo.} Set $\rho\in(0,1)$ and $\delta\in(0,1)$. Choose $t_{k}$ as the largest value in $\{\delta^{0},\delta^{1},\delta^{2},\ldots\}$ such that
\begin{equation}\label{armijio}
	F(x^{k}+t_{k} d^{k})\preceq F(x^{k})+\rho t_{k} \psi(x^{k},d^{k})e,
\end{equation}
where $d^{k}$ is the descent direction and $e=(1,1,\ldots,1)^{\top}\in\mathbb{R}^{m}$. 

\begin{remark}\setstretch{1.25}\normalfont
	Let $d^{k}=\tau^{-1}_{k}v(x^{k})$. From \cite[Lemma 4]{fliege2000steepest} and the relation $$JF(x^{k})d^{k}\preceq\psi(x^{k},d^{k})e\preceq\tau^{-1}_{k}\psi(x^{k},v(x^{k}))e\preceq0,$$ it follows that the above Armijo's stepsize strategy is well-defined.
\end{remark}

Building upon the preceding discussions, we now present the first modified version of the MSD algorithm for solving problem \eqref{mop}, denoted as MSD-I.

\begin{algorithm}\setstretch{1.25}
	\small\caption{MSD-I}\label{algo1}
	\begin{algorithmic}[1]
		\Require Initial point $x^{0}$ and $\tau_{0}=1$.
		\For{$k=0,1,2,\ldots$}
		\State Compute $\lambda^{k}$ by solving problem \eqref{dual}, and obtain the direction $v(x^{k})$ using \eqref{dual_sol};
		\If{$\|v(x^{k})\|=0$}
		\State \textbf{Return}
		Pareto critical point $x^{k}$;
		\EndIf
		\State Compute $d^{k}=\tau_{k}^{-1}v(x^{k})$, and find $t_{k}>0$ by the Armijo stepsize strategy \eqref{armijio};
		\State Update $x^{k+1}=x^{k}+t_{k}d^{k}$;
		\State Determine $\tau_{k+1}$ using \eqref{tau_k};
		\If{$\tau_{k+1}<0$}
		\State $\tau_{k+1}=1$;
		\EndIf
		\EndFor
	\end{algorithmic}
\end{algorithm}

We next discuss the convergence properties of the MSD-I algorithm. Before proceeding, we need to make some assumptions:

\begin{description}
	\item[\rm\textbf{(A1)}] $F$ is bounded below on the set $\mathcal{L}=\{x\in\mathbb{R}^{n}:F(x)\preceq F(x^{0})\}$, where $x^{0}\in\mathbb{R}^{n}$ is a given initial point.
	\item[\rm\textbf{(A2)}] Each $\nabla F_{i}$ is Lipschitz continuous with $L_{i}>0$ on an open set $\mathcal{B}$ containing $\mathcal{L}$, i.e., $\|\nabla F_{i}(x)- \nabla F_{i}(y)\|\leq L_{i}\|x-y\|$ for all $x,y\in\mathcal{B}$ and $i\in\langle m\rangle$.
	\item[\rm\textbf{(A3)}] The sequence $\{\tau_{k}\}$ has an upper bound, i.e., there exists a real number $b>0$ such that $\tau_{k}\leq b$ for all $k$.
\end{description}

The following lemma gives the value of the iterative decrease of the objective function $F$ when the MSD-I algorithm is applied. In the sequel, let $L=\max_{i\in\langle m\rangle}L_{i}$.

\begin{lemma}\setstretch{1.25}\label{mdsd_vk_bounds}
	Assume that (A2) and (A3) hold. Let $\{x^{k}\}$ be the sequence  produced by the MSD-I algorithm. Then, for all $k\geq0$, we have
	\begin{equation}\label{mdsd_vk_bounds_0}
		F(x^{k})-F(x^{k+1})\succeq \omega\|v(x^{k})\|^{2}e,
	\end{equation}
	where 
	$\omega=\min\{\rho/(2b),\rho\delta (1-\rho)/(4L)\}$.
\end{lemma}

\begin{proof}
	We have the following two cases.
	
	\emph{Case 1.} Let $\mathcal{K}_{1}=\{k:t_{k}=1\}$. From \eqref{armijio}, and taking into account the relation $\psi(x^{k},v(x^{k}))<-\|v(x^{k})\|^{2}/2$, for every $i\in\langle m\rangle$ and all $k\in\mathcal{K}_{1}$, we have
	\begin{equation}\label{mdsd_vk_bounds_1}
		\begin{aligned}
			F_{i}(x^{k})-F_{i}(x^{k+1})\geq-\rho \psi(x^{k},d^{k})
			=-\rho\tau_{k}^{-1} \psi(x^{k},v(x^{k}))
			\geq\frac{\rho}{2\tau_{k}}\|v(x^{k})\|^{2}
			\geq \frac{\rho}{2b}\|v(x^{k})\|^{2}.
		\end{aligned}
	\end{equation}
	
	\emph{Case 2.} Let $\mathcal{K}_{2}=\{k:t_{k}<1\}$. Obviously, $t_{k}/\delta\leq1$ for $k\in\mathcal{K}_{2}$. Let $\tilde{t}=t_{k}/\delta$. By the way $t_{k}$ is chosen in Armijo stepsize \eqref{armijio}, it follows that $\tilde{t}$ fails to
	satisfy \eqref{armijio}, i.e., for all $k\in\mathcal{K}_{2}$, 
	\begin{equation*}
		F(x^{k}+\tilde{t} d^{k})\npreceq F(x^{k})+\rho\tilde{t} \psi(x^{k},d^{k})e,
	\end{equation*}
	which means that
	\begin{equation*}
		F_{i_{k}}(x^{k}+\tilde{t} d^{k})- F_{i_{k}}(x^{k})>\rho \tilde{t} \psi(x^{k},d^{k})
	\end{equation*}
	for at least one index $i_{k}\in\langle m\rangle$. Using the mean-value theorem on the above inquality, there exists $\nu_{k}\in[0,1]$ such that
	\begin{equation*}\label{armijio_dk_bounds3}
		\langle\nabla F_{i_{k}}(x^{k}+\tilde{t}\nu_{k}d^{k}),d^{k}\rangle>\rho  \psi(x^{k},d^{k}).
	\end{equation*}
	This together with the definition of $\psi(\cdot,\cdot)$ gives us
	\begin{equation}\label{armijio_dk_bounds5}
		\begin{aligned}
			(\rho-1) \psi(x^{k},d^{k})&<\langle\nabla F_{i_{k}}(x^{k}+\tilde{t}\nu_{k}d^{k}),d^{k}\rangle- \psi(x^{k},d^{k})\\
			&\leq\langle\nabla F_{i_{k}}(x^{k}+\tilde{t}\nu_{k}d^{k}),d^{k}\rangle-\langle\nabla F_{i_{k}}(x^{k}),d^{k}\rangle.
		\end{aligned}
	\end{equation}
	By Cauchy-Schwarz inequality and (A2), for all $k\in\mathcal{K}_{2}$, we get
	\begin{equation}\label{armijio_dk_bounds6}
		\begin{aligned}
			\langle\nabla F_{i_{k}}(x^{k}+t\nu_{k}d^{k}),d^{k}\rangle-\langle\nabla F_{i_{k}}(x^{k}),d^{k}\rangle&\leq\|\nabla F_{i_{k}}(x^{k}+\tilde{t}\nu_{k}d^{k})-\nabla F_{i_{k}}(x^{k})\|\|d^{k}\|\\
			&\leq L_{i_{k}} \tilde{t}\nu_{k}\|d^{k}\|^{2}\\
			&\leq L \tilde{t}\|d^{k}\|^{2}.
		\end{aligned}
	\end{equation}
	Therefore, by \eqref{armijio_dk_bounds5} and \eqref{armijio_dk_bounds6}, for all $k\in\mathcal{K}_{2}$, we immediately have 
	\begin{equation*}
		t_{k}=\tilde{t}\delta\geq- \frac{\delta(1-\rho)}{L} \frac{\psi(x^{k},d^{k})}{\|d^{k}\|^{2}}.
	\end{equation*} 
	Replacing $d^{k}=\tau_{k}^{-1}v^{k}$ and observing $\psi(x^{k},v^{k})<-\|v(x^{k})\|^{2}/2$, for all $k\in\mathcal{K}_{2}$, we obtain
	\begin{equation*}\label{armijio_dk_bounds7}
		t_{k}\geq\frac{\delta(1-\rho)\tau_{k}}{2L}.
	\end{equation*} 
	Substituting this relation into \eqref{armijio}, for each $i\in\langle m\rangle$ and all $k\in\mathcal{K}_{2}$, one has
	\begin{equation}\label{armijio_dk_bounds8}
		\begin{aligned}
			F_{i}(x^{k})-F_{i}(x^{k+1})&\geq-\rho t_{k}\psi(x^{k},d^{k})\\
			&\geq-\frac{\delta\rho(1-\rho)\tau_{k}}{2L}\cdot\tau_{k}^{-1}\psi(x^{k},v(x^{k}))\\
			&\geq\frac{\delta\rho(1-\rho)}{4L}\|v(x^{k})\|^2.
		\end{aligned}
	\end{equation}
	It follows from \eqref{mdsd_vk_bounds_1} and \eqref{armijio_dk_bounds8} that the desired result \eqref{mdsd_vk_bounds_0} holds. 
\end{proof}

\begin{theorem}\setstretch{1.25}\label{convergence_analysis1}
	Assume that (A1)--(A3) hold. Let $\{x^{k}\}$ be the sequence produced by the MSD-I algorithm. Then $\lim_{k\rightarrow\infty}\|v(x^{k})\|=0$. 
\end{theorem}
\begin{proof}
	Since $F$ is bounded below, and the sequence ${F(x^{k})}$ is demonstrated to be monotonically decreasing as shown in \eqref{mdsd_vk_bounds_0}, it follows that
	$$\lim_{k\rightarrow\infty}(F(x^{k})-F(x^{k+1}))=0,$$
	which, in conjunction with \eqref{mdsd_vk_bounds_0}, leads us to the desired conclusion.
\end{proof}

\begin{remark}\setstretch{1.25}\normalfont
	A direct consequence of Theorem \ref{convergence_analysis1} and Proposition \ref{pa_sta_equ}(i) and (iii) is
	that every limit point of $\{x_{k}\}$ generated by the MSD-I algorithm is a Pareto critical point.
\end{remark}

In the rest of this section, we discuss the convergence rate of the MSD-I algorithm.

\begin{description}\setstretch{1.25}
	\item[\rm\textbf{(A4)}] There exist constants $a$ and $b$ satisfying $0<a\leq1\leq b$ such that
	\begin{equation}\label{strong_con}
		a\|z\|^{2}\leq z^{\top}\nabla^{2}F_{i}(x) z\leq b\|z\|^{2}\quad \text{for~all}~ x\in\mathcal{L}, z\in\mathbb{R}^{n}, i\in\langle m\rangle.
	\end{equation}
\end{description}

\begin{remark}\setstretch{1.25}\normalfont\label{a4}
	Assumption (A4) means that $F_{i}(\cdot)$ is strongly convex with a common modulus $a>0$ for all $i\in\langle m\rangle$. Assumption (A4) also implies assumptions (A1)--(A3). Furthermore, from \cite[Theorem 5.24]{beck2017}, it follows that
	$$(\nabla F_{i}(x)-\nabla F_{i}(y))^{\top}(x-y)\geq a\|x-y\|^{2}$$
	for all $i\in\langle m\rangle$ and $x,y\in\mathbb{R}^{n}$, and thus the inequality $a\leq L$ holds.
\end{remark}

\begin{theorem}\setstretch{1.25}
	Suppose that (A4) holds. Let $\{x^{k}\}$ be the sequence generated by the MSD-I algorithm and $\bar{x}$ be the Pareto optimal limit point of the sequence associated with the multiplier vector $\bar{\lambda}$. Then, $\{x^{k}\}$ converges R-linearly to $\bar{x}$.
\end{theorem}

\begin{proof}
	Considering the second-order Taylor expansion of $F_{i}(x)$ around $x^{k}$ and using \eqref{strong_con}, we have
	\begin{equation*}
		\dfrac{a}{2}\|x-x^{k}\|^{2}\leq F_{i}(x)-F_{i}(x^{k})-\nabla F_{i}(x^{k})^{\top}(x-x^{k})
	\end{equation*}
	for all $ i\in\langle m\rangle$. Substituting $x=\bar{x}$ in the above inequality and observing that $\lambda^{k}=(\lambda_{1}^{k},\lambda_{2}^{k},\ldots,\lambda_{m}^{k})\in\varLambda^{m}$ and \eqref{dual_sol}, we obtain
	\begin{equation}\label{rate1}
		\dfrac{a}{2}\|\bar{x}-x^{k}\|^{2}\leq \sum_{i=1}^{m}\lambda_{i}^{k}F_{i}(\bar{x})-\sum_{i=1}^{m}\lambda_{i}^{k}F_{i}(x^{k})+v(x^{k})^{\top}(\bar{x}-x^{k}).
	\end{equation}
	According to Lemma \ref{mdsd_vk_bounds}, we have that $F_{i}(\bar{x})\leq F_{i}(x^{k})$ for all $i\in\langle m\rangle$. Therefore,  $$\sum_{i=1}^{m}\lambda_{i}^{k}F_{i}(\bar{x})-\sum_{i=1}^{m}\lambda_{i}^{k}F_{i}(x^{k})\leq0.$$
	This, combined with \eqref{rate1} and the Cauchy-Schwartz inequality, yields
	\begin{equation*}
		\dfrac{a}{2}\|\bar{x}-x^{k}\|^{2}\leq v(x^{k})^{\top}(\bar{x}-x^{k})\leq\|v(x^{k})\|\|\bar{x}-x^{k}\|,
	\end{equation*}
	which implies that 
	\begin{equation}\label{rate2}
		\dfrac{a}{2}\|\bar{x}-x^{k}\|\leq\|v(x^{k})\|.
	\end{equation}
	On the other hand, 	if we consider the second-order Taylor expansion of $F_{i}(x)$ around $\bar{x}$ and employing \eqref{strong_con}, then
	\begin{equation*}
		\dfrac{a}{2}\|x-\bar{x}\|^{2}\leq F_{i}(x)-F_{i}(\bar{x})-\nabla F_{i}(\bar{x})^{\top}(x-\bar{x})\leq \dfrac{b}{2}\|x-\bar{x}\|^{2}
	\end{equation*}
	for all $ i\in\langle m\rangle$. Letting $x=x^{k}$ in the above inequality, and observing that $\bar{\lambda}=(\bar{\lambda}_{1},\bar{\lambda}_{2},\ldots,\bar{\lambda}_{m})\in\varLambda^{m}$ and \eqref{dual_sol} as well as Proposition \ref{pa_sta_equ}(i), we get
	\begin{equation}\label{rate3}
		\dfrac{a}{2}\|x^{k}-\bar{x}\|^{2}\leq \sum_{i=1}^{m}\bar{\lambda}_{i}F_{i}(x^{k})-\sum_{i=1}^{m}\bar{\lambda}_{i}F_{i}(\bar{x})\leq \dfrac{b}{2}\|x^{k}-\bar{x}\|^{2}.
	\end{equation}
	It follows from \eqref{rate2} and \eqref{rate3} that
	\begin{equation}\label{rate4}
		\sum_{i=1}^{m}\bar{\lambda}_{i}F_{i}(x^{k})-\sum_{i=1}^{m}\bar{\lambda}_{i}F_{i}(\bar{x})\leq \dfrac{2b}{a^{2}}\|v(x^{k})\|^{2}.
	\end{equation}
	By \eqref{mdsd_vk_bounds_0}, we immediately have
	\begin{equation*}
		\sum_{i=1}^{m}\bar{\lambda}_{i}F_{i}(x^{k+1})\leq \sum_{i=1}^{m}\bar{\lambda}_{i}F_{i}(x^{k})- \omega\|v(x^{k})\|^{2},
	\end{equation*}
	By subtracting the term $\sum_{i=1}^{m}\bar{\lambda}_{i}F_{i}(\bar{x})$ from both sides of the above inequality and taking into account \eqref{rate4}, yields
	\begin{equation}\label{rate5}
		\begin{aligned}
			\sum_{i=1}^{m}\bar{\lambda}_{i}F_{i}(x^{k+1})-\sum_{i=1}^{m}\bar{\lambda}_{i}F_{i}(\bar{x})&\leq \sum_{i=1}^{m}\bar{\lambda}_{i}F_{i}(x^{k})-\sum_{i=1}^{m}\bar{\lambda}_{i}F_{i}(\bar{x})- \omega\|v(x^{k})\|^{2},\\
			&\leq\left(1-\dfrac{wa^{2}}{2b}\right)\left(\sum_{i=1}^{m}\bar{\lambda}_{i}F_{i}(x^{k})-\sum_{i=1}^{m}\bar{\lambda}_{i}F_{i}(\bar{x})\right).
		\end{aligned}
	\end{equation}
	According to the setting of $\omega$ as in Lemma \ref{mdsd_vk_bounds} and Remark \ref{a4}, it is easy to see that $wa^{2}/2b\in(0,1)$. Therefore, recursively applying \eqref{rate5}, we have
	\begin{equation*}
		\begin{aligned}
			\sum_{i=1}^{m}\bar{\lambda}_{i}F_{i}(x^{k+1})-\sum_{i=1}^{m}\bar{\lambda}_{i}F_{i}(\bar{x})\leq\left(1-\dfrac{wa^{2}}{2b}\right)^{k+1}\left(\sum_{i=1}^{m}\bar{\lambda}_{i}F_{i}(x^{0})-\sum_{i=1}^{m}\bar{\lambda}_{i}F_{i}(\bar{x})\right),
		\end{aligned}
	\end{equation*}
	which, combined with the left-hand side of \eqref{rate3}, gives
	\begin{equation*}
		\begin{aligned}
			\|x^{k+1}-\bar{x}\|\leq\left(\dfrac{2}{a}\left(\sum_{i=1}^{m}\bar{\lambda}_{i}F_{i}(x^{0})-\sum_{i=1}^{m}\bar{\lambda}_{i}F_{i}(\bar{x})\right)\right)^{1/2}\left(\left(1-\dfrac{wa^{2}}{2b}\right)^{1/2}\right)^{k+1}.
		\end{aligned}
	\end{equation*}
	This means that $\{x^{k}\}$ converges R-linearly to $\bar{x}$.
\end{proof}

\section{The second improvement scheme}\label{sec:4}

In this section, we present the second method for choosing the parameter $\theta_{k}$, inspired by Andrei's work \cite{andrei2006acceleration}. For each $i\in\langle m\rangle$, assume that $F_{i}$ is convex and  twice continuously differentiable. Let us consider the second-order Taylor expansion of each function $F_{i}$ at the point $x^{k}+t_{k}v(x^{k})$, i.e.,
\begin{equation}\label{I1}
	\begin{aligned}
		F_{i}(x^{k}+t_{k}v(x^{k}))=&~ F_{i}(x^{k})+t_{k}\nabla F_{i}(x^{k})^{\top}v(x^{k})\\
		&+\dfrac{1}{2}t_{k}^{2}v(x^{k})^{\top}\nabla^{2}F_{i}(x^{k})v(x^{k})+o(\|t_{k}v(x^{k})\|^{2}).
	\end{aligned}
\end{equation}
Similarly, for $\theta>0$ and each $i\in\langle m\rangle$, we have
\begin{equation}\label{I2}
	\begin{aligned}
		F_{i}(x^{k}+\theta t_{k}v^{k})=&\, F_{i}(x^{k})+\theta t_{k}\nabla F_{i}(x^{k})^{\top}v(x^{k})\\
		&+\dfrac{1}{2}\theta^{2}t_{k}^{2}v(x^{k})^{\top}\nabla^{2}F_{i}(x^{k})v(x^{k})+o(\|\theta t_{k}v(x^{k})\|^{2}).
	\end{aligned}
\end{equation}
By \eqref{I1} and \eqref{I2}, we have
\begin{equation*}\label{}
	\begin{aligned}
		F_{i}(x^{k}+\theta t_{k}v(x^{k}))=&\, F_{i}(x^{k}+t_{k}v^{k})+(\theta-1) t_{k}\nabla F_{i}(x^{k})^{\top}v(x^{k})\\
		&+\dfrac{1}{2}(\theta^{2}-1)t_{k}^{2}v(x^{k})^{\top}\nabla^{2}F_{i}(x^{k})v(x^{k})\\
		&+o(\|\theta t_{k}v(x^{k})\|^{2})
		-o(\| t_{k}v(x^{k})\|^{2}).
	\end{aligned}
\end{equation*}
Multiplying by $\lambda_{i}^{k}$ and summing over  $i\in\langle m\rangle$ in the above equation, and observing that the relation \eqref{dual_sol}, one has
\begin{equation}\label{I3}
	\sum_{i=1}^{m}\lambda_{i}^{k}F_{i}(x^{k}+\theta t_{k}v(x^{k}))=\sum_{i=1}^{m}\lambda_{i}^{k}F_{i}(x^{k}+ t_{k}v(x^{k}))+\varphi_{k}(\theta),
\end{equation}
where 
\begin{equation*}
	\begin{aligned}
		\varphi_{k}(\theta)=\,&(1-\theta) t_{k}\|v(x^{k})\|^{2}+\dfrac{1}{2}(\theta^{2}-1)t_{k}^{2}v(x^{k})^{\top}\left(\sum_{i=1}^{m}\lambda_{i}^{k}\nabla^{2}F_{i}(x^{k})\right)v(x^{k})\\
		&+(\theta^{2} -1)t_{k}o(t_{k}\| v(x^{k})\|^{2}).
	\end{aligned}
\end{equation*}
Denote
\begin{align}
	p_{k} &= t_{k}\|v(x^{k})\|^{2}, \label{eq:pk} \\
	q_{k} &= t_{k}^{2}v(x^{k})^{\top}\left(\sum_{i=1}^{m}\lambda_{i}^{k}\nabla^{2}F_{i}(x^{k})\right)v(x^{k}) \label{eq:qk}
\end{align}
and $\epsilon_{k}=o(t_{k}\| v(x^{k})\|^{2})$. Since $F_{i}$ is convex and $\lambda_{i}^{k}\geq0$ for all $i\in\langle m\rangle$, then $q_{k}\geq0$. An equivalent expression of $\varphi_{k}(\theta)$ is
\begin{equation*}
	\varphi_{k}(\theta)=\left(\dfrac{1}{2}q_{k}+t_{k}\epsilon_{k}\right)\theta^{2}-p_{k}\theta+\left(p_{k}-\dfrac{1}{2}q_{k}-t_{k}\epsilon_{k}\right),\quad \theta>0.
\end{equation*}
Observe that $$\varphi'_{k}(\theta)=(q_{k}+2t_{k}\epsilon_{k})\theta-p_{k}$$ and $\varphi'_{k}(0)=-p_{k}<0$. Therefore, the optimal solution of $\varphi_{k}(\theta)$ is
\begin{equation}\label{I3_1}
	\bar{\theta}=\dfrac{p_{k}}{q_{k}+2t_{k}\epsilon_{k}}.
\end{equation}
The optimal value of $\varphi_{k}(\theta)$ at the point $\bar{\theta}$ is 
\begin{align}\label{I3_2}
	\varphi_{k}(\bar{\theta})=-\dfrac{(p^{k}-(q^{k}+2t_{k}\epsilon_{k}))^{2}}{2q^{k}+4t_{k}\epsilon_{k}}\leq0.
\end{align}
From \eqref{I3} and \eqref{I3_2}, and letting $\theta_{k}=\bar{\theta}$, we get
\begin{equation}\label{I3_3}
	\begin{aligned}
		\sum_{i=1}^{m}\lambda_{i}^{k}F_{i}(x^{k}+\theta_{k} t_{k}v(x^{k}))&=\sum_{i=1}^{m}\lambda_{i}^{k}F_{i}(x^{k}+ t_{k}v(x^{k}))+\varphi_{k}(\theta_{k})\\
		&\leq\sum_{i=1}^{m}\lambda_{i}^{k}F_{i}(x^{k}+ t_{k}v(x^{k})).
	\end{aligned}
\end{equation}
This implies that the value $\sum_{i=1}^{m}\lambda_{i}^{k}F_{i}(x^{k}+\theta_{k} t_{k}v(x^{k}))$ has a possible improvement.

At the point $x^{k}$, by the Armijo stepsize strategy \eqref{armijio}, we can find a $t_{k}\in(0,1]$ such that \eqref{armijio} with $d^{k}=v(x^{k})$ holds. Furthermore, by $\psi(x^{k},v(x^{k}))<-\|v(x^{k})\|^{2}/2$, we have
\begin{equation*}
	\begin{aligned}
		\sum_{i=1}^{m}\lambda^{k}_{i}F_{i}(x^{k}+t_{k} v(x^{k}))
		&\leq \sum_{i=1}^{m}\lambda^{k}_{i}F_{i}(x^{k})+\rho t_{k} \psi(x^{k},v(x^{k}))\\
		&\leq\sum_{i=1}^{m}\lambda^{k}_{i}F_{i}(x^{k})-\dfrac{\rho t_{k}}{2}\|v(x^{k})\|^{2},
	\end{aligned}
\end{equation*}
which, combined with \eqref{I3_3}, yields
\begin{equation}\label{I3_4}
	\begin{aligned}
		\sum_{i=1}^{m}\lambda_{i}^{k}F_{i}(x^{k}+\theta_{k} t_{k}v(x^{k}))&\leq\sum_{i=1}^{m}\lambda^{k}_{i}F_{i}(x^{k})-\left[\dfrac{\rho p_{k}}{2}+\dfrac{(p^{k}-(q^{k}+2t_{k}\epsilon_{k}))^{2}}{2q^{k}+4t_{k}\epsilon_{k}}\right]\\
		&\leq \sum_{i=1}^{m}\lambda^{k}_{i}F_{i}(x^{k}).
	\end{aligned}
\end{equation}
If we neglect the contribution of $\epsilon_{k}$ in the above relation, then an improvement of the function value $\sum_{i=1}^{m}\lambda_{i}^{k}F_{i}(x^{k}+\theta_{k} t_{k}v(x^{k}))$ can be obtained as 
\begin{equation*}
	\begin{aligned}
		\sum_{i=1}^{m}\lambda_{i}^{k}F_{i}(x^{k}+\theta_{k} t_{k}v(x^{k}))
		&\leq\sum_{i=1}^{m}\lambda^{k}_{i}F_{i}(x^{k})-\left[\dfrac{\rho p_{k}}{2}+\dfrac{(p^{k}-q^{k})^{2}}{2q^{k}}\right]
		\leq \sum_{i=1}^{m}\lambda^{k}_{i}F_{i}(x^{k}).
	\end{aligned}
\end{equation*}

Because the definition of $q_{k}$ involves Hessian information, this presents a challenge for practical computation. Now, we give a way for computing $q_{k}$. Considering the second-order Taylor expansion of each function $F_{i}$ ($i\in\langle m\rangle$) at the point $z^{k}=x^{k}+t_{k}v(x^{k})$, we get
\begin{equation*}
	\begin{aligned}
		F_{i}(z^{k})= F_{i}(x^{k})+t_{k}\nabla F_{i}(x^{k})^{\top}v(x^{k})
		+\dfrac{1}{2}t_{k}^{2}v(x^{k})^{\top}\nabla^{2}F_{i}(\beta_{1})v(x^{k}),
	\end{aligned}
\end{equation*}
where $\beta_{1}\in\mathbb{R}^{n}$ lies between $x^{k}$ and $x^{k+1}$. Multiplying by $\lambda_{i}^{k}$ and summing over $i\in\langle m\rangle$ in the above equation, one has
\begin{equation}\label{I4}
	\begin{aligned}
		\sum_{i=1}^{m}\lambda_{i}^{k}F_{i}(z^{k})=& \sum_{i=1}^{m}\lambda_{i}^{k}F_{i}(x^{k})+t_{k}\left(\sum_{i=1}^{m}\lambda_{i}^{k}\nabla F_{i}(x^{k})\right)^{\top}v(x^{k})\\
		&+\dfrac{1}{2}t_{k}^{2}v(x^{k})^{\top}\left(\sum_{i=1}^{m}\lambda_{i}^{k}\nabla^{2}F_{i}(\beta_{1})\right)v(x^{k}).
	\end{aligned}
\end{equation}
Similarly, at the point $x^{k}=z^{k}-t_{k}v(x^{k})$, we get
\begin{equation*}
	\begin{aligned}
		F_{i}(x^{k})= F_{i}(z^{k})-t_{k}\nabla F_{i}(z^{k})^{\top}v^{k}
		+\dfrac{1}{2}t_{k}^{2}v(x^{k})^{\top}\nabla^{2}F_{i}(\beta_{2})v(x^{k}),
	\end{aligned}
\end{equation*}
where $\beta_{2}\in\mathbb{R}^{n}$ is between $x^{k}$ and $z^{k}$. We can further obtain
\begin{equation}\label{I5}
	\begin{aligned}
		\sum_{i=1}^{m}\lambda_{i}^{k}F_{i}(x^{k})=& \sum_{i=1}^{m}\lambda_{i}^{k}F_{i}(z^{k})-t_{k}\left(\sum_{i=1}^{m}\lambda_{i}^{k}\nabla F_{i}(z^{k})\right)^{\top}v(x^{k})\\
		&+\dfrac{1}{2}t_{k}^{2}v(x^{k})^{\top}\left(\sum_{i=1}^{m}\lambda_{i}^{k}\nabla^{2}F_{i}(\beta_{2})\right)v(x^{k}).
	\end{aligned}
\end{equation}
Using the local character of searching and then considering the distance between $z^{k}$ and $x^{k}$ is small enough, we can take $\beta_{1}=\beta_{2}=x^{k}$. Therefore, by \eqref{I4}, \eqref{I5} and the definition of $q_{k}$ as in \eqref{eq:qk},  we obtain
\begin{equation}\label{qk_new}
	q_{k}=t_{k}\left(\sum_{i=1}^{m}\lambda_{i}^{k}(\nabla F_{i}(z^{k})-\nabla F_{i}(x^{k}))\right)^{\top}v(x^{k}).
\end{equation}
Observe that the computation of $q_{k}$ needs an additional evaluation of the gradient at the point $z^{k}$. If we ignore the influence of $\epsilon_{k}$ in \eqref{I3_1}, then $\theta_{k}=p_{k}/q_{k}$. 

After providing some discussions, we are now ready to describe the second modified version of the MSD algorithm for solving problem \eqref{mop}, which is called MSD-II.

\begin{algorithm}[H]\setstretch{1.25}
	\small\caption{MSD-II}\label{algo2}
	\begin{algorithmic}[1]
		\Require Initial point $x^{0}$.
		\For{$k=0,1,2,\ldots$}
		\State Compute $\lambda^{k}$ by solving problem \eqref{dual}, and obtain the direction $v(x^{k})$ using \eqref{dual_sol};
		\If{$\|v(x^{k})\|=0$}
		\State \textbf{Return} Pareto critical point $x^{k}$;
		\EndIf
		\State Find $t_{k}>0$ by the Armijo stepsize strategy \eqref{armijio} with $d^{k}=v(x^{k})$;
		\State Compute $z^{k}=x_{k}+t_{k}v(x^{k})$;
		\State Compute $p^{k}$ and $q^{k}$ according to \eqref{eq:pk} and \eqref{qk_new}, respectively;
		\State Compute $\theta_{k}=p_{k}/q_{k}$;
		\State Update $x^{k+1}=x^{k}+\theta_{k}t_{k}d^{k}$;
		%\Else
		%\State Set $x^{k+1}=z^{k}$;
		%\EndIf
		\EndFor
	\end{algorithmic}
\end{algorithm}

\begin{remark}\setstretch{1.25}\label{rem5}\normalfont
	Note that the algorithm presented in \cite{andrei2006acceleration} exhibits linear convergence with a factor smaller than that of the classical SD algorithm.
	As for the MSD-II algorithm, we currently lack convergence results similar to those in \cite{andrei2006acceleration}, but its numerical effectiveness is evident in Section \ref{sec:5}. Therefore, we pose an open question in this paper: How to obtain the convergence and the rate of convergence of the MSD-II algorithm under suitable assumptions? 
\end{remark}

\section{Numerical experiments}\label{sec:5}

This section reports the results of numerical experiments, comparing the performance of the proposed algorithms with the MSD algorithm \cite[Algorithm 1]{fliege2000steepest} and multi-objective diagonal steepest descent (MDSD) algorithm \cite[Algorithm 1]{el2021accelerated}. All numerical experiments were implemented in MATLAB R2020b and executed on a personal desktop (Intel Core Duo i7-12700,  2.10 GHz, 32 GB of RAM). The optimization problem for determining the descent direction, e.g., problem \eqref{dual}, was solved by adopting the  conditional gradient algorithm \cite{beck2017}. 

In order to make the comparisons as fair as possible, we used the same line search strategy presented in \eqref{armijio} for all algorithms. The parameter values for the Armijo line search \eqref{armijio} were chosen as follows: $\rho=10^{-4}$ and $\delta=0.5$. In the MDSD algorithm, the initial parameter $\tau_{0}$ was set to $10^{-4}$. All runs were stopped whenever $\lvert\gamma(x^{k})\rvert$ is less than or equal $10^{-6}$. Since Proposition \ref{pa_sta_equ} implies that $v(x) = 0$ if and only if $\gamma(x) = 0$, this stopping criterion makes sense. The maximum number of allowed outer iterations was set to 1000.

We selected 32 test problems from the miltiobjective optimization literature, as shown in Table \ref{testpro}. The first two columns identify the problem name and the corresponding reference. ``Convex'' reports whether the problem is convex or not. Columns ``$m$'' and ``$n$'' inform the number of objective functions and the number of variables of the problems, respectively. The initial points were generated within a box defined by the lower and upper bounds, denoted as $x_{L}$ and $x_{U}$, respectively, as indicated in the last two columns. 
It is important to emphasize that the boxes presented in Table \ref{testpro} were exclusively used for specifying the initial points and were not considered by the algorithms during their execution. Note that MGH33 and TOI4 are adaptations of two single-objective optimization problems to the multi-objective setting that can be found in \cite{mita2019nonmonotone}. For the WIT suites, WIT1--WIT6  indicate the parameters in WIT are set to  $\lambda=0,0.5,0.9,0.99,0.999,1$, respectively (see \cite[Example 4.2]{witting2012numerical}). From a fairness point of view, the compared algorithms use the same 100 initial points on each given test problem, which are selected uniformly from the given upper and lower bounds. In Section \ref{sec:4}, we assume that the objective function $F$ is convex, which ensures that the term $q_{k}$ in \eqref{eq:qk} remains non-negative.  However, for non-convex objective functions, it is possible for $\theta_{k}=p_{k}/q_{k}<0$ to occur in the MSD-II algorithm. In such cases, our numerical experiments require setting $\theta_{k}=1$.

\begin{table}[htbp]\scriptsize
	\centering\small
	\renewcommand\arraystretch{1.5}
	\caption{List of Test Problems.}
		\begin{tabular}{lllllll}
			\hline
			Problem & Source  & Convex & $m$ & $n$ & $x_{L}$     &  $x_{U}$ \\
			\hline
			AP2 &   \cite{ansary2015modified}  &$\checkmark$ &  2  & 1    & $-100$ &100\\
			AP4 &   \cite{ansary2015modified} &$\checkmark$ &  3   & 3    & $(-10,-10,-10)$ & (10,10,10) \\
			BK1 &    \cite{huband2006review} &$\checkmark$ & 2   & 2    & $(-5,-5)$ & (10,10) \\
			DGO1 &    \cite{huband2006review} & $\times$ & 2  & 1    & $-10$ & 13\\
			DGO2 &    \cite{huband2006review}& $\checkmark$  & 2  & 1    & $-9$ & 9\\
			Far1 &    \cite{huband2006review} &$\times$  & 2  & 2    & $(-1,-1)$ & (1,1)\\
			FDS&    \cite{fliege2009newton}  & $\checkmark$  & 3 & 10    & $(-2,-2,...,-2)$ & (2,2,...,2)\\
			FF1 &    \cite{huband2006review} &$\times$  & 2  & 2    & $(-1,-1)$ & (1,1)\\
			Hil1 &    \cite{hil} & $\times$ & 2  & 2    & $(0,0)$ & (1,1)\\
			JOS1a&    \cite{jin2001dynamic} & $\checkmark$ & 2  & 50    & $(-100,-100,...,-100)$ & (100,100,...,100) \\
			JOS1b&    \cite{jin2001dynamic} & $\checkmark$ & 2  & 100    & $(-100,-100,...,-100)$ & (100,100,...,100) \\
			JOS1c&    \cite{jin2001dynamic} & $\checkmark$ & 2  & 1000    & $(-100,-100,...,-100)$ & (100,100,...,100) \\
			JOS1d&    \cite{jin2001dynamic} & $\checkmark$ & 2  & 5000    & $(-100,-100,...,-100)$ & (100,100,...,100) \\
			KW2 &   \cite{lovison2011singular} &$\times$   & 2 &2 &$ (-3,-3)$ & (3,3) \\
			Lov1 &   \cite{lovison2011singular} &$\checkmark$   & 2 &2 &$ (-10,-10)$ & (10,10) \\
			Lov3 &   \cite{lovison2011singular} & $\times$  & 2 &2 &$ (-20,-20)$ & (20,20) \\
			Lov4 &   \cite{lovison2011singular} & $\times$  & 2 &2 &$ (-20,-20)$ & (20,20) \\
			MGH33  &   \cite{more1981test} & $\checkmark$ &  10 & 10  &$ (-1,-1,...,-1)$ & (1,1,...,1) \\
			MHHM2 &   \cite{huband2006review}&$\checkmark$&  3    & 2  &$ (0,0)$ & (1,1) \\
			MLF1 &   \cite{huband2006review}&$\times$&  2    & 1  &0 &20 \\
			MLF2 &   \cite{huband2006review}&$\times$&  2    & 2  &$ (-100,-100)$ & (100,100) \\
			MMR1 &   \cite{mmr}&$\times$&  2    & 2  &$ (0.1,0)$ & (1,1) \\
			MOP3 &   \cite{mmr}&$\times$&  2    & 2  &$ (-\pi,-\pi)$ &$ (\pi,\pi)$ \\
			PNR &   \cite{preuss2006pareto} &$\checkmark$&  2   & 2   & $(-2,-2)$ & (2,2)\\
			SP1&    \cite{huband2006review}  & $\checkmark$& 2  & 2   & $(-100,-100)$ & (100,100)\\
			TOI4 &   \cite{toint1983test} &  $\checkmark$&  2 & 4   & $(-2,-2,-2,-2)$ & (2,2,2,2)\\
			WIT1  &    \cite{witting2012numerical}&$\checkmark$ &2   & 2     & $(-2,-2)$ & (2,2) \\
			WIT2  &   \cite{witting2012numerical}  & $\checkmark$ &2& 2     & $(-2,-2)$ & (2,2) \\
			WIT3  &  \cite{witting2012numerical} &$\checkmark$ &2   & 2     & $(-2,-2)$ & (2,2) \\
			WIT4 &   \cite{witting2012numerical}  &$\checkmark$&2  & 2     & $(-2,-2)$ & (2,2) \\
			WIT5  &   \cite{witting2012numerical}& $\checkmark$ &2  & 2     & $(-2,-2)$ & (2,2) \\
			WIT6&    \cite{witting2012numerical} &$\checkmark$ &2  & 2     & $(-2,-2)$ & (2,2) \\
			\hline
	\end{tabular}
	\label{testpro}%
\end{table}

Table \ref{results} summarizes the results obtained by Algorithms 1 and 2, comparing them with the MSD and MDSD algorithms. The table is organized into columns labeled ``it", ``fE", ``gE", ``T" and "\%''. The column ``it'' denotes the average number of iterations, while ``fE'' and ``gE'' stand for the average number of function and gradient evaluations, respectively. The column ``T'' means the average computational time (in seconds) to reach the critical point from an initial point, and ``\%'' indicates the percentage of runs that has reached a critical point. As observed in Table \ref{results}, failures can occur with the MSD algorithm when attempting to solve AP4, MMR1 and TOI4. In particular, for large-scale JOS1, the MSD algorithm struggles to converge within the maximum number of iterations. In comparison to the other three algorithms, the MDSD algorithm exhibits unsatisfactory performance on some problems like AP4 and FDS. However, its performance improves moderately and significantly over the MSD algorithm for most problems. The newly proposed MSD-I algorithm outperforms the MSD algorithm and demonstrates competitiveness with the MDSD algorithm. Notably, the MSD-II algorithm exhibits a significant advantage over the MSD algorithm, as well as the MDSD algorithm. Since WIT suites are controlled by the parameter $\lambda$, they fall into the category of scaled problems, potentially impacting algorithmic performance. Nevertheless, our algorithms perform well, especially the MSD-II algorithm, which exhibits robustness. The last column of Table \ref{results} represents the sum of the numerical results. As we can see, the MDSD algorithm outperforms the MSD algorithm in terms of ``it'' and ``gE'', but lags behind in ``fE'' and ``T''. This discrepancy could be attributed to the less robust initial parameters $\tau_{0}$ of the MDSD algorithm. Among all algorithms, the MSD-II algorithm emerges as the top-performing one, surpassing the MSD, MDSD and MSD-I algorithms in terms of ``it", ``fE", ``gE" and ``T". The MSD-I algorithm stands as the second-best one, outperforming the MSD and MDSD algorithms. In conclusion, the introduction of the adjusted parameter $\theta_{k}$ in this paper proves to be an effective tool for enhancing the convergence of the MSD algorithm.

\begin{sidewaystable}[htbp]
	\centering\setstretch{1.25}
	\caption{Performance of the proposed algorithms in comparison with the MSD and MDSD algorithms on the chosen set of test problems.}
	\resizebox{\linewidth}{!}{
		\begin{tabular}{lllllllllllllllllllll}
			\hline
			&         & \multicolumn{4}{l}{MSD}              & \multicolumn{5}{l}{MDSD}                     & \multicolumn{5}{l}{MSD-I}                    & \multicolumn{5}{l}{MSD-II}               \\\hline
			Problem & it      & fE      & gE      & T          & \%  & it     & fE      & gE     & T          & \%  & it     & fE      & gE     & T          & \%  & it   & fE     & gE    & T          & \%  \\\hline
			AP2     & 0.98    & 2.94    & 1.98    & 3.7508e-04 & 100 & 1.94   & 17.6    & 2.94   & 6.9159e-04 & 100 & 0.98   & 2.94    & 1.98   & 3.5444e-04 & 100 & 0.98 & 2.94   & 2.96  & 3.6132e-04 & 100 \\
			AP4     & 190.92  & 579.43  & 191.84  & 4.2999e-01 & 92  & 248.99 & 2419.44 & 249.99 & 2.4880e+01 & 100 & 158.87 & 1522.5  & 159.82 & 1.8599e-01 & 100 & 3.1  & 9.36   & 7.2   & 4.2519e-02 & 100 \\
			BK1     & 1       & 3       & 2       & 3.2900e-04 & 100 & 2      & 18      & 3      & 8.1037e-04 & 100 & 1      & 3       & 2      & 9.9512e-05 & 100 & 1    & 3      & 3     & 1.0717e-04 & 100 \\
			DGO1    & 2.29    & 4.58    & 3.29    & 6.0702e-04 & 100 & 1.22   & 9.03    & 2.22   & 4.8855e-04 & 100 & 2.21   & 4.42    & 3.21   & 1.9069e-04 & 100 & 1.56 & 3.12   & 4.12  & 1.5716e-04 & 100 \\
			DGO2    & 46.69   & 93.39   & 47.69   & 1.1373e-02 & 100 & 3.34   & 16.97   & 4.34   & 1.4392e-03 & 100 & 4.14   & 8.29    & 5.14   & 2.9462e-04 & 100 & 2.38 & 4.77   & 5.76  & 2.1747e-04 & 100 \\
			Far1    & 33.31   & 124.34  & 34.31   & 8.0649e-03 & 100 & 47.98  & 330.22  & 48.98  & 1.5936e-02 & 100 & 30.32  & 164.47  & 31.32  & 2.1340e-03 & 100 & 4.91 & 14.68  & 10.82 & 4.9329e-04 & 100 \\
			FDS     & 91.91   & 382.13  & 92.91   & 9.2474e+00 & 100 & 111    & 780.02  & 112    & 4.1956e+01 & 100 & 77.71  & 502.87  & 78.71  & 3.7979e+00 & 100 & 3.97 & 11.7   & 8.94  & 1.5460e-01 & 100 \\
			FF1     & 22.99   & 45.98   & 23.99   & 5.7183e-03 & 100 & 28.97  & 206.73  & 29.97  & 3.5952e-03 & 100 & 16.3   & 91.72   & 17.3   & 4.3484e-03 & 100 & 9.6  & 19.2   & 20.2  & 7.7916e-04 & 100 \\
			Hil1    & 9.19    & 42.25   & 10.19   & 9.1944e-04 & 100 & 6.88   & 43.13   & 7.88   & 8.2422e-04 & 100 & 7.58   & 28.71   & 8.58   & 1.7837e-03 & 100 & 3.71 & 16.05  & 8.42  & 2.9402e-04 & 100 \\
			JOS1a   & 229.64  & 459.28  & 230.64  & 6.7849e-02 & 100 & 2      & 12      & 3      & 4.6603e-04 & 100 & 2      & 4       & 3      & 7.0467e-04 & 100 & 1    & 2      & 3     & 1.8919e-04 & 100 \\
			JOS1b    & 861.94  & 1723.88 & 862.94  & 1.9745e-01 & 100 & 2      & 10      & 3      & 1.2870e-03 & 100 & 2      & 4       & 3      & 9.4530e-04 & 100 & 1    & 2      & 3     & 9.3505e-04 & 100 \\
			JOS1c    & 1000    & 2000    & 1000    & 4.6299e-01 & 100 & 2      & 8       & 3      & 8.4976e-04 & 100 & 2      & 4       & 3      & 7.2835e-04 & 100 & 1    & 2      & 3     & 6.0991e-04 & 100 \\
			JOS1d    & 1000    & 2000    & 1000    & 8.4425e-01 & 100 & 1      & 4       & 2      & 1.1221e-03 & 100 & 2      & 4       & 3      & 1.3631e-03 & 100 & 1    & 2      & 3     & 1.1348e-03 & 100 \\
			KW2     & 37.81   & 98.35   & 38.81   & 1.0758e-02 & 100 & 6.03   & 41.04   & 7.03   & 1.0176e+00 & 100 & 9.83   & 30.97   & 10.83  & 3.3048e-03 & 100 & 6.68 & 19.89  & 14.36 & 8.7466e-04 & 100 \\
			Lov1    & 2.98    & 8.74    & 3.98    & 8.1300e-04 & 100 & 2.79   & 19.58   & 3.79   & 3.4262e-04 & 100 & 2.84   & 6.48    & 3.84   & 7.3155e-04 & 100 & 2.18 & 6.14   & 5.36  & 2.0484e-04 & 100 \\
			Lov3    & 4.6     & 13.76   & 5.6     & 1.1896e-03 & 100 & 4.66   & 27.14   & 5.66   & 6.6703e-04 & 100 & 3.43   & 8.62    & 4.43   & 5.7345e-04 & 100 & 2.16 & 6.44   & 5.32  & 2.2877e-04 & 100 \\
			Lov4    & 1.36    & 4.07    & 2.36    & 3.7940e-04 & 100 & 4.19   & 27.57   & 5.19   & 7.1047e-04 & 100 & 1.27   & 3.62    & 2.27   & 4.2700e-04 & 100 & 1.11 & 3.32   & 3.22  & 1.1675e-04 & 100 \\
			MGH33   & 3       & 34.51   & 4       & 1.0540e-03 & 100 & 1.57   & 25.67   & 2.57   & 2.8003e-04 & 100 & 1.85   & 12.98   & 2.85   & 6.8230e-04 & 100 & 1    & 11.14  & 3     & 1.7039e-04 & 100 \\
			MHHM2   & 1       & 3       & 2       & 6.4357e-04 & 100 & 1.73   & 17.46   & 2.73   & 3.7389e-01 & 100 & 1      & 3       & 2      & 6.2975e-04 & 100 & 1    & 3      & 3     & 2.1667e-04 & 100 \\
			MLF1    & 0.58    & 1.16    & 1.58    & 1.9128e-04 & 100 & 0.4    & 2.33    & 1.4    & 3.5504e-04 & 100 & 0.58   & 1.16    & 1.58   & 3.6755e-04 & 100 & 0.79 & 1.58   & 2.58  & 7.8369e-05 & 100 \\
			MLF2    & 45.2    & 158.78  & 46.2    & 1.0854e-02 & 100 & 36.8   & 191.3   & 37.8   & 3.6271e-03 & 100 & 29.35  & 128.58  & 30.35  & 6.6559e-03 & 100 & 9.37 & 30.98  & 19.74 & 7.1169e-04 & 100 \\
			MMR1    & 15.12   & 44.77   & 16.11   & 3.7141e-03 & 99  & 3.52   & 25.09   & 4.52   & 5.7945e-04 & 100 & 3.9    & 12      & 4.9    & 5.9721e-04 & 100 & 3.26 & 10.67  & 7.52  & 3.3148e-04 & 100 \\
			MOP3    & 7.55    & 28.43   & 8.55    & 2.0552e-03 & 100 & 6.23   & 32.02   & 7.23   & 6.3641e-04 & 100 & 6.14   & 16.03   & 7.14   & 1.5225e-03 & 100 & 4.57 & 17.22  & 10.14 & 3.8289e-04 & 100 \\
			PNR     & 5.03    & 25.68   & 6.03    & 1.4600e-03 & 100 & 5.46   & 30.14   & 6.46   & 7.9231e-04 & 100 & 4.63   & 15.05   & 5.63   & 6.7535e-04 & 100 & 1.35 & 5.61   & 3.7   & 1.5677e-04 & 100 \\
			SP1     & 10.93   & 38.65   & 11.93   & 2.8252e-03 & 100 & 10.27  & 37.89   & 11.27  & 1.0384e-03 & 100 & 9.67   & 23.85   & 10.67  & 2.3389e-03 & 100 & 9.29 & 29.31  & 19.58 & 7.7312e-04 & 100 \\
			Toi4    & 183.42  & 367.96  & 184.32  & 4.2969e-02 & 90  & 3.24   & 20.31   & 4.24   & 5.7318e-04 & 100 & 3.54   & 7.18    & 4.54   & 5.7948e-04 & 100 & 2.37 & 5.4    & 5.74  & 2.2346e-04 & 100 \\
			WIT1    & 42      & 318.57  & 43      & 1.0899e-02 & 100 & 76.75  & 479.1   & 77.75  & 7.8088e-03 & 100 & 37.43  & 227.23  & 38.43  & 9.0981e-03 & 100 & 1.47 & 7.05   & 3.94  & 1.4425e-04 & 100 \\
			WIT2    & 67.98   & 516.23  & 68.98   & 1.7592e-02 & 100 & 77.16  & 473.85  & 78.16  & 7.7801e-03 & 100 & 61.18  & 360.88  & 62.18  & 1.4813e-02 & 100 & 1.85 & 10.99  & 4.7   & 1.8832e-04 & 100 \\
			WIT3    & 25.95   & 152.92  & 26.95   & 6.6296e-03 & 100 & 27.5   & 143.12  & 28.5   & 2.8450e-03 & 100 & 27.71  & 130.04  & 28.71  & 6.6641e-03 & 100 & 2.13 & 10.79  & 5.26  & 2.0369e-04 & 100 \\
			WIT4    & 6.26    & 26.1    & 7.26    & 1.6755e-03 & 100 & 7.22   & 34.69   & 8.22   & 7.4005e-04 & 100 & 5.39   & 15.94   & 6.39   & 1.3328e-03 & 100 & 1.91 & 7.23   & 4.82  & 1.9090e-04 & 100 \\
			WIT5    & 5.04    & 17.93   & 6.04    & 1.3009e-03 & 100 & 4.45   & 25.31   & 5.45   & 4.6433e-04 & 100 & 4.5    & 12.33   & 5.5    & 1.1216e-03 & 100 & 1.9  & 6.43   & 4.8   & 1.9835e-04 & 100 \\
			WIT6    & 1       & 3       & 2       & 3.1636e-04 & 100 & 1.99   & 17.98   & 2.99   & 2.6246e-04 & 100 & 1      & 3       & 2      & 3.1964e-04 & 100 & 1    & 3      & 3     & 1.7745e-04 & 100 \\ \hline
			Total & 3957.67 & 9323.81 & 3987.48 & 1.1395e+01 &     & 741.28 & 5546.73 & 773.28 & 6.8285e+01 &     & 522.35 & 3363.86 & 554.3  & 4.0493e+00 &     & 90.6 & 289.01 & 213.2 & 2.0797e-01 &    \\
			\hline
	\end{tabular}}
	\label{results}%
\end{sidewaystable}%

Next, we used the performance profile\footnote{The performance profiles in this paper were generated using the \texttt{MATLAB} code \texttt{perfprof.m}, which is freely available on the website \url{https://github.com/higham/matlab-guide-3ed/blob/master/perfprof.m}.} proposed in \cite{D_b2002} by Dolan and Mor\'{e} to analyze the performance of various algorithms. The performance profile has become a commonly used tool for assessing the performance of multiple solvers $\mathcal{S}$ when applied to a test set $\mathcal{P}$ in scalar optimization. It is worth noting that this tool is also utilized in multi-objective optimization; see \cite{mita2019nonmonotone,cocchi2020convergence,chen2023memory,lapucci2023limited,morovati2018quasi}. We provide a brief description of the tool. Let there be $n_{s}$ solvers and $n_{p}$ problems.  For $s\in\mathcal{S}$ and $p\in\mathcal{P}$, denote $o_{p,s}$ as the performance of solver $s$ on problem $p$. The performance ratio is defined as $z_{p,s}=o_{p,s}/\min\{o_{p,s}:s\in\mathcal{S}\}$ and the cumulative distribution function $\rho_{s}:[1,\infty)\rightarrow[0,1]$ is given by
$$\rho_{s}(\tau)=\frac{\left\lvert\left\{p\in\mathcal{P}:z_{p,s}\leq\tau\right\}\right\rvert}{n_{p}}.$$
The performance profile is presented by depicting the cumulative distribution function $\rho_{s}$. Notably, $\rho_{s}(1)$ represents the probability of the solver outperforming the remaining solvers. The right side of the image for $\rho_{s}$ illustrates the robustness associated with a solver. Fig. \ref{iter_feval_perfprof} displays the performance profiles for the number of iterations, the number of function evaluations, the number of gradient evaluations and the computational time. The values on the $x$-axis are converted to a $\log$ scale. As can be seen in Fig. \ref{iter_feval_perfprof}, our algorithms, especially the MSD-II algorithm, showcase unparalleled performance, outperforming both the MSD and MDSD algorithms. 

\begin{figure}[htbp]
	\centering
	\subfigure[Iterations]{
		\label{iter}
		\includegraphics[width=0.4\textwidth, keepaspectratio]{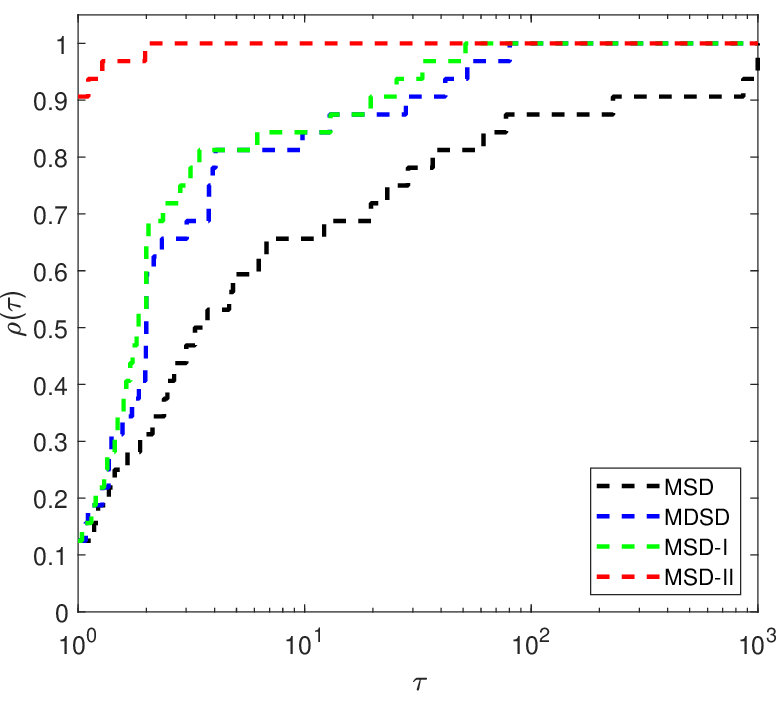}}
	\subfigure[Function evaluations]{
		\label{feval}
		\includegraphics[width=0.4\textwidth, keepaspectratio]{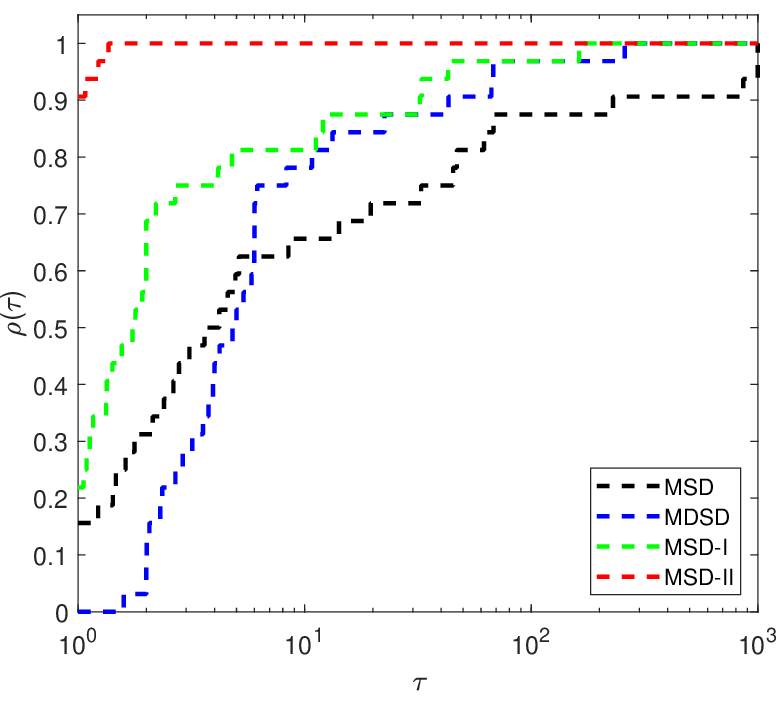}}
	\subfigure[Gradient evaluations]{
		\label{geval}
		\includegraphics[width=0.4\textwidth, keepaspectratio]{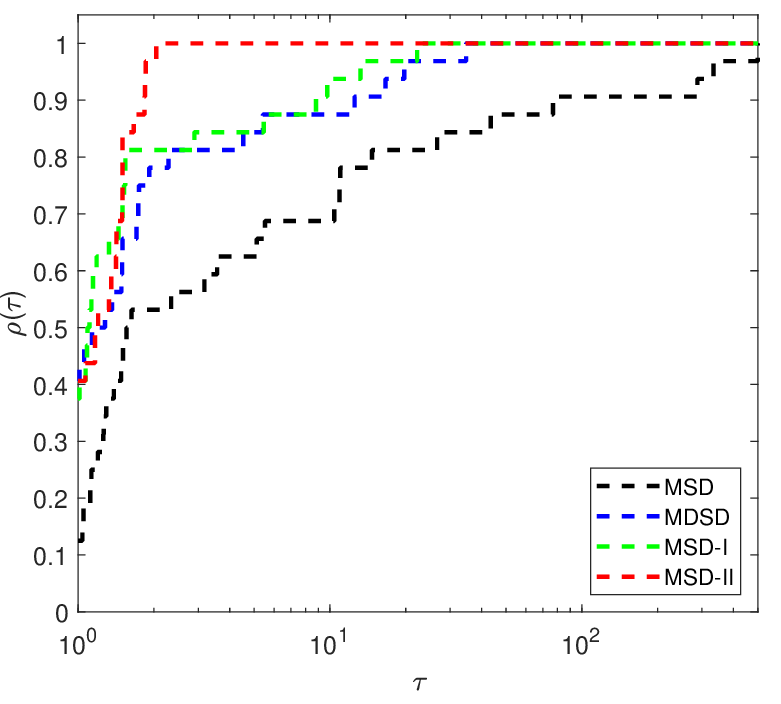}}
	\subfigure[Computational time]{
		\label{cpu}
		\includegraphics[width=0.4\textwidth, keepaspectratio]{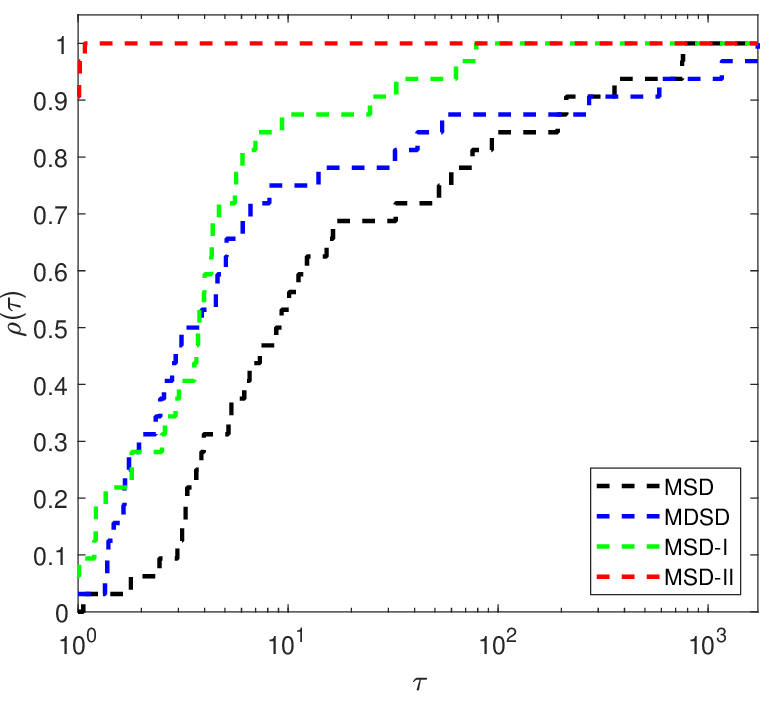}}
	\caption{Performance profiles considering 100 initial points for each test problem using as the performance measurement: (a) number of iterations; (b) number of functions evaluations; (c) number of gradient evaluations; (d) computational time.}
	\label{iter_feval_perfprof}
\end{figure}

We additionally evaluated the algorithms' capability to appropriately generate Pareto frontiers. To achieve this, we investigated seven bi-objective problems: Far1, Hil1, PNR, and WIT1--WIT4. The results illustrated in Fig. \ref{pof} demonstrate that, with a reasonable number of starting points, these methods are proficient in discovering favorable solutions for the aforementioned problems.

\begin{figure}[htbp]
	\centering
	\includegraphics[width=\textwidth,keepaspectratio]{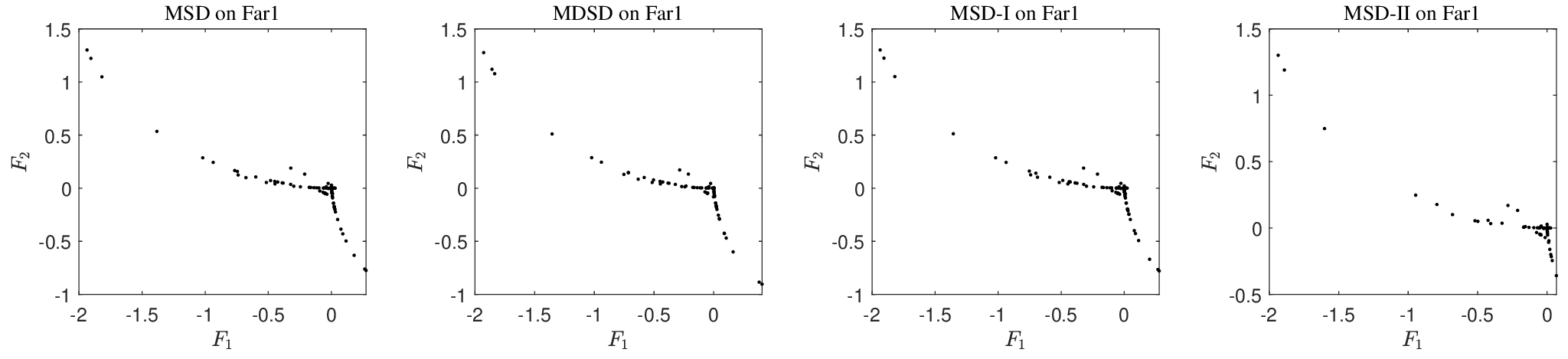}
	\includegraphics[width=\textwidth,keepaspectratio]{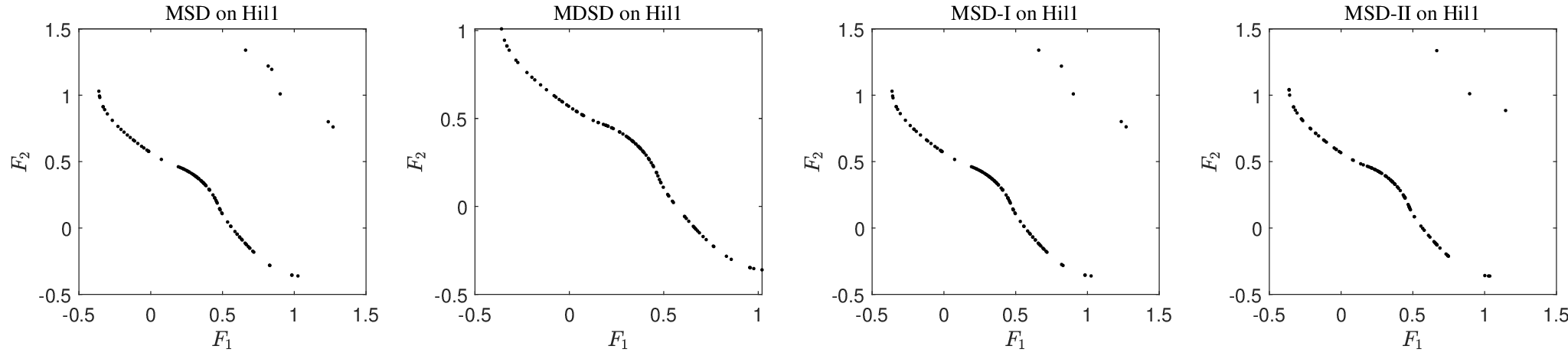}
	\includegraphics[width=\textwidth,keepaspectratio]{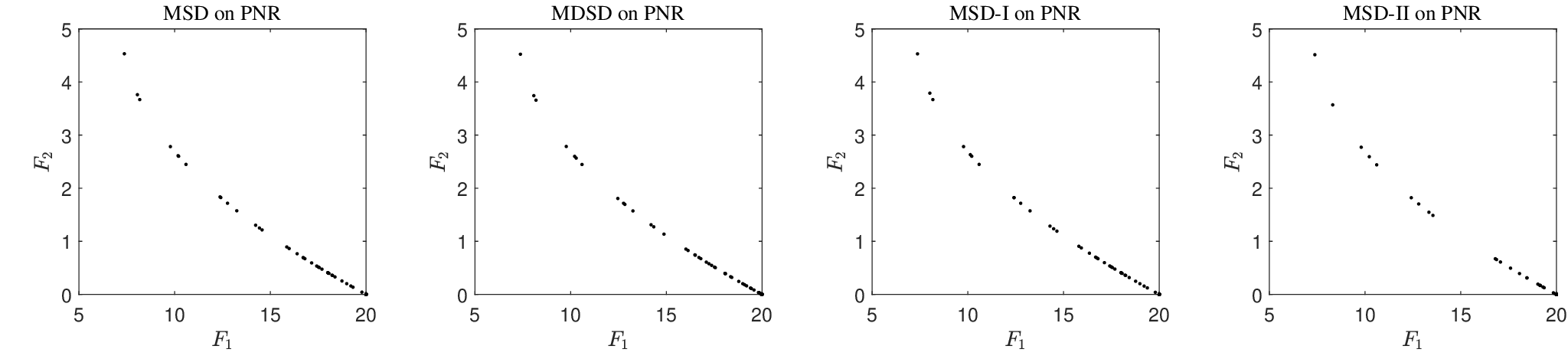}	\includegraphics[width=\textwidth,keepaspectratio]{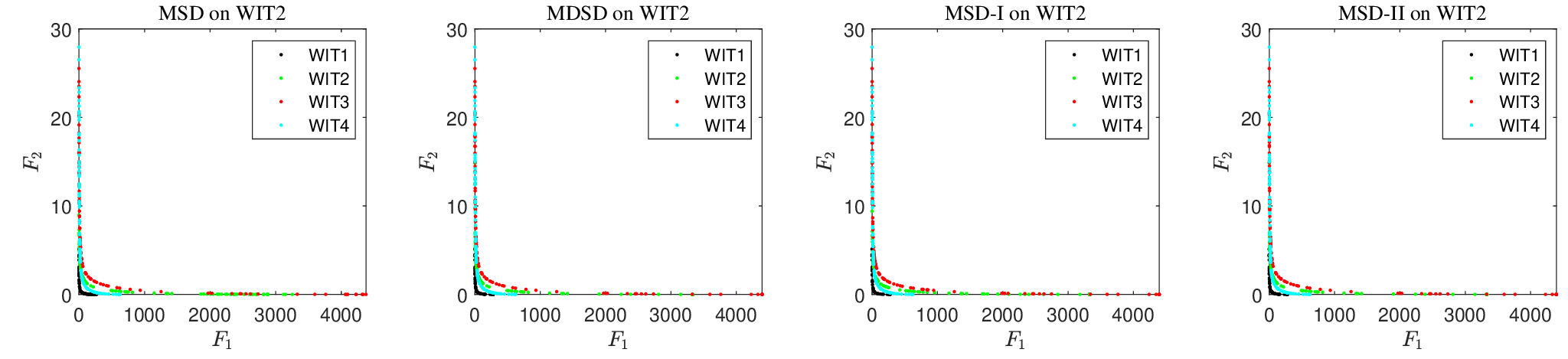}
	\vskip0.1in
	\caption{The final solutions found by the MSD, MDSD, MSD-I and MSD-II algorithms on Far1, Hil1, PNR and WIT2--WIT4.}
	\label{pof}
\end{figure}

Finally, given the excellent performance of the MSD-II algorithm as shown in Table \ref{results}, we tested it on the larger and complex convex problem FDS, varying the dimensions. As reported in \cite{fliege2009newton}, the numerical difficulty of problem FDS sharply increases with the dimension $n$. From Table \ref{fds}, it is evident that for FDS with varying dimensions, the MSD-II algorithm performs well.

\begin{table}[htbp]\scriptsize
	\centering\setstretch{1.25}
	\caption{Performance of the MSD-II algorithm on FDS, varying the dimensions.}
	\begin{tabular}{llllll}
		\hline
	    $n$    & it    & fE    & gE    & T      & \% \\\hline
		200  & 7.34  & 58.68 & 15.68 & 2.0197e+00 & 100  \\
		500  & 7.66  & 63.92 & 16.32 & 2.6750e+00  & 100  \\
		1000 & 10.06 & 89.48 & 21.12 & 2.8643e+00 & 100  \\
		2000 & 8.79  & 68.31 & 18.58 & 3.1071e+00 & 100  \\
		4000 & 9.36  & 75.11 & 19.72 & 3.7364e+00 & 100  \\
		5000 & 9.49  & 69.4  & 19.98 & 3.8682e+00 & 100\\
		10000 & 10.05  &70.70  & 21.10 &4.3852e+00 & 100\\\hline
	\end{tabular}
	\label{fds}
\end{table}

\section{Conclusion}\label{sec:6}

In this paper, we proposed a modified scheme for multi-objective steepest descent algorithm, utilizing a straightforward multiplicative adjustment of the stepsize as illustrated in \eqref{new_iter}. We presented two methods for choosing the modification parameter $\theta_{k}$, leading to two improved multi-objective steepest descent algorithms designed for solving unconstrained multi-objective optimization problems. Numerical experiments conducted on a set of unconstrained multi-objective test problems indicate that the new algorithms exhibit greater advantages compared to the multi-objective steepest descent algorithm. 

As part of our future research agenda, we aim to explore alternative strategies for updating the modification parameter $\theta_{k}$ in \eqref{new_iter}. Additionally, similar to the research work in \cite{andrei2010accelerated}, the presented multiplicative manner can be applied to different multi-objective descent algorithms, such as multi-objective conjugate gradient method \cite{lucambio2018nonlinear}, to upgrade existing performance characteristics.  While we do not provide convergence results for the MSD-II algorithm (see Remark \ref{rem5}), numerical results in Section \ref{sec:5} indicate its outstanding performance, motivating us to further investigate its theoretical properties, which will be a focus of our future work.

\end{document}